\documentclass[11pt, a4paper]{article}
\usepackage{graphicx}
\usepackage{amssymb}
\usepackage{stmaryrd}
\usepackage{enumitem}

\usepackage{geometry}
\usepackage{tikz}
\usepackage{float}
\usepackage{url}
\usepackage{mathtools}
\usepackage{amsmath,amsfonts,amsthm}
\usepackage{mathrsfs}

\usepackage{pgfplots}
\pgfplotsset{compat=1.18}

\newtheorem{theorem}{Theorem}[section]
\newtheorem{lemma}[theorem]{Lemma}
\newtheorem{proposition}{Proposition}[section]
\newtheorem{corollary}[theorem]{Corollary}
\newtheorem{definition}{Definition}

\newtheorem{conjecture}{Conjecture}

\usepackage{subfig}

\DeclarePairedDelimiter\ceil{\lceil}{\rceil}
\DeclarePairedDelimiter\floor{\lfloor}{\rfloor}

\newcommand{\CC}{\mathcal{CT}}

\usepackage[pagewise]{lineno}

\usepackage{ytableau}
\usepackage[colorlinks,
linkcolor=blue,
anchorcolor=blue,
citecolor=blue
]{hyperref}
\begin{document}

\begin{center}
{\large \bf Optimal Behaviour in Extremal Bounds for $\sigma$-Irregularity}
\end{center}
\begin{center}
 Jasem Hamoud$^1$ \hspace{0.2 cm}  Duaa Abdullah$^{2}$\\[8pt]
 $^{1,2}$ Moscow Institute of Physics and Technology\\[6pt]
Email: 	 $^{1}${\tt hamoud.math@gmail.com},
 $^{2}${\tt duaa1992abdullah@gmail.com}
\end{center}
\noindent
\begin{abstract}
In this paper, we establish extremal bounds for the Sigma index $\sigma(G) = \sum_{uv \in E(G)} (d_G(u) - d_G(v))^2$, focusing on sharp upper and lower bounds via a non-increasing degree sequence $\mathscr{D} = (d_1 \geq d_2 \geq \dots \geq d_n)$. A fundamental challenge in studying topological indices is deriving precise bounds, as such results reveal intrinsic relationships among various indices.

\end{abstract}

\noindent\textbf{AMS Classification 2010:} 05C05, 05C12, 05C20, 05C25, 05C35, 05C76, 68R10.

\noindent\textbf{Keywords:} Topological indices, Extremal, Irregularity, Sigma index, Bounds, Trees.

\section{Introduction}\label{sec1}

Throughout this paper. Let $G=(V,E)$ be a simple, connected graph, where $V(G)=\{v_1,v_2,\dots,v_n\}$, $|n=|V(G)|$ and $E(G)=\{e_1,e_2,\dots,e_m\}$, $m=|E(G)|$.
Let $\mathscr{D}=(d_1,d_2,\dots,d_{n-1},d_n)$ be a degree sequence  where $d_w=d_G(v_w)$ the degree of vertex $v_w\in V(G)$ and consider $\mathscr{D}$ is non-increasing. 
S.~L.~Hakimi~\cite{Hakimi1962} had provided the concept of realizability as a set of integers with degrees of the vertices of a graph. 
Among~\cite{Caro2014Pepper} had presented a study of degree sequence index strategy.

In~\cite{Schmidt1976Druffel,Baskin1987Gordeeva} had presented the methodology for solving the inverse problem of topological indexes. One of this topological indexes called the \emph{Albertson index} was first introduced in 1997 by M.~O.~Albertson~\cite{albertson1997irregularity}. I.~Gutman~\cite{Gutman2018I} had presented topological indices and irregularity measures where many equivalent formulas for Albertson index considered as the concept of \emph{irregularity} were presented in~\cite{abdo2014total,Abdo2018Dimitrov,Albalahi2023Alanazi,Alex2024Indulal,Akg2019Naca,Arif2023Hayat}. 
Generally, Albertson index is defined as
\[
\operatorname{irr} (G)=\sum_{uv\in E(G)}|d_G(u)-d_G(v)|.
\]
This measure is not the only possible. In fact, R.~Criado et al.~\cite{Criado2014Flores} emphasize that irregularity index is
\[
\operatorname{irr}(G)=\frac{1}{n}\sum_{i=1}^{n}\left|  d_i-\frac{2m}{n}\right|.
\]
Z.~Che and Z.~Chen~\cite{Che2016chenZZ} had established new mathematical lower and upper bounds for the \emph{forgotten topological index} $F(G)$ where the forgotten topological index had provided by Furtula and Gutman~\cite{Furtula2015Gutman} and defined of a molecular graph $G$ as
\[
F(G)=\sum_{uv\in E(G)}\left( d_G(u)^2+d_G(v)^2\right).
\]
In~\cite{Fajtlowicz1987,Mateji2018Milovanovi,Sigarreta2020jM,Vukicevic2009Furtula} had established that the \emph{harmonic topological index} $H(G)$. It inherently prioritizes edges linking vertices of lower bound, as the ratio attains greater magnitude when the aggregate degree sum is minimal. Z.~Lingping~\cite{Lingping2012} defined $H(G)$ as
\[
H(G)=\sum_{uv\in E(G)}\frac{2}{d_G(u)+d_G(v)}.
\]
The rational extension of the quantum harmonic oscillator and exceptional Hermite polynomials had presented in~\cite{Gómez2013Grandati}. 
Topological indices study of molecular structure~\cite{Gao2016WangMR} such that one of the topological indices which was associated with irregularity~\cite{Abdo2015Dimitrov} in the graph theory called  \emph{Sigma index}, denoted by $\sigma$. 
Relationships between other topological indices had presented through~\cite{Jahanbai2021Sheikholeslami} and Sigma index introduced by~\cite{gutman2018inverse,Albalahi2023Alanazi,Jahanbanı2019Ediz,Ascioglu2018Cangul} , which is defined as
\[
\sigma(G)=\sum_{uv\in E(G)}(d_G(u)-d_G(v))^2.
\]
Closely related to Sigma index had introduced by~\cite{Criado2014Flores} as
\[
\sigma(G)=\frac{1}{n}\sum_{i=1}^{n}\left( d_i-\frac{2m}{n}\right)^2.
\]
There are many studies that have examined the bounds on the Albertson index, some of which have been related to the sigma index, such as~\cite{Abdo2018Dimitrov,abdo2014total,abdo2019graph,albertson1997irregularity, Gutman2018I,Arif2023Hayat, Chen2026Liu}. Through~\cite{Ascioglu2018Cangul} provided a study of Sigma index and forgotten index. In~\cite{Chartrand2010Okamoto} had discussed specific issues of Sigma called \emph{sigma chromatic number}. The \emph{$\sigma_t$-irregularity} (or sigma total index) of a graph $G$ is defined~\cite{Knor2025RQW,Dimitrov2023Stevanović} as
\[
\sigma_t(G) = \sum_{\{u,v\} \subseteq V(G)} (d_G(u)-d_G(v))^2,
\]
some results on $\sigma_ {t} $-irregularity had given through~\cite{Filipovski2024Dimitrov}. This index has been studied in the context of extremal graph theory, yielding bounds and characterizations of graphs maximizing $\sigma_t$-irregularity.

A fundamental challenge in the study of topological indices lies in establishing precise bounds, as such findings illuminate intrinsic relationships among diverse indices. Given that topological indices serve as mathematical descriptors closely linked to the chemical attributes of compounds, pinpointing the compound exhibiting extremal behavior for a particular property is tantamount to the problem of optimizing a strongly correlated topological index within the designated class of graphs.

The main goal of this paper is to study the optimal behavior of the Sigma index by evaluating its upper and lower bounds along with some elicited topological indices. Since many important studies have been derived from the Albertson index and the Sigma index, we see it as important to determine the closest bounds for these indices.

This paper is organized as follows. In Section~\ref{sec3}, we presented the upper bound of Sigma index for a degree sequence $\mathscr{D}$ order such that $d_1\geqslant d_2 \geqslant \dots \geqslant d_n$. Section~\ref{sec4} discussed the lower bounds of Sigma index among degree sequence. In Section~\ref{sec5}, we presented the effects of this bounds. 
 \section{Preliminaries}\label{sec2}
In this section, we introduce several essential concepts to ensure the reader fully comprehends the results presented in this paper. Several studies have introduced the concept of bounds on topological indices~\cite{Oboudi2019MR,Sarasija2017Binthiya}. Therefore, we will review some bounds on topological indices that contribute to strengthening our study regarding extremal bounds. 
\begin{definition}[Zagreb Indices~\cite{gutman1972total, gutman1975acyclic,Vasilyev2014Darda,Yang2021Arockiaraj}]
Let $ G=(V, E) $ be a graph. The first and second Zagreb indices are given by
\[
M_1(G) = \sum_{v \in V(G)} d_G(v)^2, \quad M_2(G) = \sum_{uv \in E(G)} d_G(u)\,.\,d_G(v).
\]
\end{definition}
In the following Theorem~\ref{thmforggotn1}, we provide the relationship with the second Zagreb index, the forgotten indx and Sigma index.
\begin{theorem}[\cite{Ascioglu2018Cangul}]~\label{thmforggotn1}
For any connected graph $G$ we have $\sigma(G)=F(G)-2M_2(G)$.
\end{theorem}
During our study of the Sigma index through this paper, we will not use the traditional definition but will instead use the relation obtained from Theorem~\ref{basic01sigma}, which is associated with degree sequence $\mathscr{D}$ among caterpillar tree $\CC$. We consider caterpillar trees denoted by $\CC(n, m)$, where $n$ is the number of backbone (or path) vertices and $m$ is the number of pendant vertices attached to each.  

\begin{theorem}[\cite{HamoudwithDuaa}]~\label{basic01sigma}
Let $\CC$ be a caterpillar tree with a degree sequence
$\mathscr{D}=(d_1,\dots,d_n),$ where $d_1\geqslant d_2\geqslant \dots
\geqslant d_n$. Then, Sigma index among caterpillar tree $\CC$ is
\[
\sigma(\CC)=(d_n-1)^3+(d_1-1)^3+\sum_{i=1}^{n}(d_i-d_{i+1})^2+\sum_{i=2}^{n-1}(d_i-1)^2(d_i-2).
\]
\end{theorem}
For a duplicate star graph $\mathcal{S}_{r,k}$, Theorem~\ref{thmpirlllkj} provide us the relationship of Sigma index with special terms.
\begin{theorem}[\cite{gutman2018inverse}]~\label{thmpirlllkj}
Let $\mathcal{S}_{r,k}$ be a duplicate star graph where $\deg_{\mathcal{S}_{r,k}}(u)=k$ and $d_{\mathcal{S}_{r,k}}(v)=r$. Then, 
\[
\sigma(\mathcal{S}_{r,k})=(k-1)^3+(r-1)^3+(k-r)^2.
\]
\end{theorem}

Let $\mathcal{T}_{n,p}$ be a tree of $n$ vertices and $p$ pendent vertices. Consider the caterpillar tree $\CC\in \mathcal{T}_{n,p}$. In this study we refer to caterpillar tree  $\CC$ by considering Figure~\ref{figTypCater01}.
\begin{figure}[H]
    \centering
\begin{tikzpicture}[scale=.8]
\draw   (1,2)-- (3,2);
\draw   (3,2)-- (5,2);
\draw   (7,2)-- (9,2);
\draw [line width=1pt,dotted] (5,2)-- (7,2);
\draw [line width=1pt,dotted] (3,2)-- (3,0);
\draw [line width=1pt,dotted] (5,2)-- (5,0);
\draw [line width=1pt,dotted] (7,2)-- (7,0);
\draw [line width=1pt,dotted] (9,2)-- (9,0);
\draw [line width=1pt,dotted] (1,2)-- (1,0);
\draw (0.8399938252334064,2.999767639650283) node[anchor=north west] {$d_1$};
\draw (2.816853283710542,2.974744102201205) node[anchor=north west] {$d_2$};
\draw (4.918830429433066,2.9246970273030497) node[anchor=north west] {$d_3$};
\draw (6.745548663215736,2.974744102201205) node[anchor=north west] {$d_{n-1}$};
\draw (8.722408121692872,3.0498147145484378) node[anchor=north west] {$d_n$};
\draw (9.097761183429038,0.6) node[anchor=north west] {$p$};
\draw (7.1709487998500565,0.6) node[anchor=north west] {$p$};
\draw (5.119018729025687,0.6) node[anchor=north west] {$p$};
\draw (3.1421592705485515,0.6) node[anchor=north west] {$p$};
\draw (0.2,0.6) node[anchor=north west] {$p$};
\begin{scriptsize}
\draw [fill=black] (1,2) circle (1.5pt);
\draw [fill=black] (3,2) circle (1.5pt);
\draw [fill=black] (5,2) circle (1.5pt);
\draw [fill=black] (7,2) circle (1.5pt);
\draw [fill=black] (9,2) circle (1.5pt);
\end{scriptsize}
\end{tikzpicture}
    \caption{Caterpillar tree $\CC(n,p)$.}
    \label{figTypCater01}
\end{figure}
In next lemma, X.~Sun et al.~\cite{Sun2025DuJMei} had provided an important upper bound of Sigma index. It establishes the following section.
\begin{lemma}[\cite{Sun2025DuJMei}]
A tree $T$ of order $n \ge 5$ with $p$ pendent vertices, where $3 \leq p \leq n-2$, satisfies $\sigma(T) \leq (p-1)^{3} + (p-2)^{2} + 1.$
Equality holds if and only if $T \cong \mathcal{T}_{n,p}$.
\end{lemma}

\section{Upper Bound of Sigma Index among Caterpillar Trees}\label{sec3}
In this section, the upper bound of the Sigma index is characterized by encompassing several bounds. 
We observe through Proposition~\ref{powPronu1} a confinement of the Sigma index within the interval $[\frac{1}{3}\, , \frac{1}{10}]$ by considering the term $2n^2-1$. Proposition~\ref{powPronu2} presented the constant upper bound related to $n^3+m^2$.

\begin{proposition}~\label{powPronu1}
Let $\CC$ be a caterpillar tree, $\mathscr{D}=(d_1,d_2,\dots,d_n)$ be a degree sequence where $d_1\geqslant d_2 \geqslant \dots \geqslant d_n$. Then, Sigma index satisfy
\begin{equation}~\label{eq1powPronu1}
3(n^2-1)-t_0+(d_{n-1}-d_n)^3\leqslant \sigma(\CC)\leqslant 10(n^2-1)-t_0+(d_{n-1}-d_n)^3,
\end{equation}
where $t_0=2d_1^3-3d_1^2d_2+3d_1d_2^2-d_1^3$.
\end{proposition}
\begin{proof}
Assume $\mathscr{D}=(d_1,d_2,\dots,d_n)$ is a degree sequence with $d_1\geqslant d_2\geqslant\dots\geqslant d_n$ and maximum degree $\Delta>3$. Since $n>\Delta$ and $2n^3+m>4\Delta$, the expression $\frac{2n^3+m}{4\Delta}$ yields the Sigma index satisfying relation~\eqref{eq2powPronu1}, where $\Delta\geqslant d_1$:
\begin{equation}\label{eq2powPronu1}
\sigma(\CC)\geqslant\floor*{\frac{2n^3+m}{4\Delta}}.
\end{equation}
Let $k>0$ be an integer. By relation~\eqref{eq1powPronu1}, it holds that $k\leqslant n^2-\frac{1}{2}\leqslant k+1$. For $j\geqslant k$, relation~\eqref{eq2powPronu1} implies $4\Delta j+4\Delta\leqslant 2n^3+m\leqslant 4\Delta j.$ Thus,~\eqref{eq2powPronu1} yields
\begin{equation}\label{eq3powPronu1}
\floor*{\frac{2n^3+m}{4\Delta}}-\floor*{\frac{2n^2-1}{2}}\leqslant 4\Delta(k-j).
\end{equation}
Hence, $\sigma(\CC)>d_1^3+(d_1-d_2)^3$. For the term $d_1^3+(d_1-d_2)^3$, the inequality $(d_{n-1}-d_n)^3<d_1^3+(d_1-d_2)^3<\sum_{i=1}^nd_i^3$
implies $\sigma(\CC)\geqslant d_1^3+(d_1-d_2)^3-(d_{n-1}-d_n)^3$. Therefore,~\eqref{eq3powPronu1} gives
\begin{equation}\label{eq4powPronu1} 
   \sigma(\CC)\geqslant\frac{d_1^3+(d_1-d_2)^3-(d_{n-1}-d_n)^3}{4\Delta(k-j)}.
\end{equation}
In particular,
\begin{equation}\label{eq5powPronu1}
1\leqslant\frac{d_1^3+(d_1-d_2)^3-(d_{n-1}-d_n)^3}{4\Delta(k-j)}\leqslant 3.
\end{equation}
The relations~\eqref{eq4powPronu1} and~\eqref{eq5powPronu1}, together with the bound on the Sigma index in~\eqref{eq5powPronu1}, imply that~\eqref{eq1powPronu1} holds.
\end{proof}

\begin{proposition}~\label{powPronu2}
Let $\CC$ be a caterpillar tree and $\mathscr{D}=(d_1,d_2,\dots,d_n)$ be a degree sequence where $d_1\geqslant d_2 \geqslant \dots \geqslant d_n$ with the maximum degree $\Delta$. Then,
\begin{equation}~\label{eq1powPronu2}
0<\frac{1}{m-1}\left(\frac{2n^3+m^2}{\Delta(\Delta-1)^2}+\floor*{\frac{3nm}{7}}-\ceil*{\frac{3nm}{7}}\right)<3,
\end{equation}
if and only if $\sqrt{\sigma(\CC)}>n$.
\end{proposition}
\begin{proof}
Consider the degree sequence $\mathscr{D}=(d_1,d_2,\dots,d_n)$ where $d_1\geqslant d_2 \geqslant \dots \geqslant d_n$ and the maximum degree is $\Delta$, with $4\Delta \leqslant n$. Since $m=n-1$, we have $2n^3+m^2=2n^3+n^2-2n+1$, which implies $2n^3+n^2-2n+1>\Delta(\Delta-1)^2$. For certain values of $\Delta$, this yields a lower bound for the Sigma index.
\begin{equation}~\label{eq2powPronu2}
\sigma(\CC)\geqslant \frac{2n^3+n^2-2n+1}{\Delta(\Delta-1)^2}+2n.
\end{equation}
Thus, from~\eqref{eq2powPronu2} we have $\sigma(\CC)>2n$ and $\sqrt{\sigma(\CC)}>n$. Hence, in what follows we work under the assumption that the Sigma index satisfies equation~\eqref{eq1powPronu2}. We now examine the bounds and their relation to the Sigma index, and then establish how these bounds are connected to the Sigma index via~\eqref{eq1powPronu2}. Then,
\begin{equation}~\label{eq3powPronu2}
\sigma(\CC)\leqslant \floor*{\frac{3nm}{7}}+\ceil*{\frac{3nm}{7}}+2n^2+2,
\end{equation}
where $2n^2+2$ is the constant term. Since $\sigma(\CC)>2n^2+2$ implies that $\sigma(\CC)>\floor*{\frac{3nm}{7}}+\ceil*{\frac{3nm}{7}}.$ In this case, let $3<k<n/2$ and $k\in \mathbb{N}$. Then, according to~\eqref{eq3powPronu2} the lower bound of Sigma index satisfy
\begin{equation}~\label{eq4powPronu2}
\sigma(\CC)\geqslant \floor*{\frac{knm}{k+4}}+\ceil*{\frac{knm}{k+4}}.
\end{equation}
Thus, according to both lower bounds~\eqref{eq3powPronu2} and \eqref{eq4powPronu2} we find that 
\begin{align*}
\sigma(\CC) & \geqslant \floor*{\frac{knm}{k+4}}+\ceil*{\frac{knm}{k+4}}\\
& \geqslant \frac{1}{m-1} \left ( \floor*{\frac{knm}{k+4}}+\ceil*{\frac{knm}{k+4}} \right) \\
& \geqslant \frac{1}{m-1} \left ( \floor*{\frac{knm}{k+4}}+\ceil*{\frac{knm}{k+4}} +2n^2+2\right).
\end{align*}
Therefore, the lower bounds of Sigma index had discussed among~\eqref{eq4powPronu2} and \eqref{eq2powPronu2}. Thus, according to Proposition~\ref{powPronu1} we find that $\sigma(\CC)\geqslant \floor*{\frac{2n^3+m}{4\Delta}}$. Thus, it holds 
\begin{equation}~\label{eq5powPronu2}
n<\frac{1}{m-1} \left ( \floor*{\frac{knm}{k+4}}+\ceil*{\frac{knm}{k+4}} +2n^2+2\right)<n\Delta
\end{equation}
Since $2n^2+2$ is the constant term and  $\sigma(\CC)>2n^2+2$ according to~\eqref{eq5powPronu2} noticed that 
\begin{equation}~\label{eq6powPronu2}
 2n<\frac{2n^3+n^2-2n+1}{\Delta(\Delta-1)^2}<n\Delta.
\end{equation}
Thus, when $d_1\geqslant d_2 \geqslant \dots \geqslant d_n,$ we have $2n/(m-1)\approx 2$ and $n\Delta/(m-1)\leqslant n^2/\Delta$. Thus, 
\begin{equation}~\label{eq7powPronu2}
 0<\frac{1}{m-1}\left(\frac{2n^3+n^2-2n+1}{\Delta(\Delta-1)^2}\right)<3.
\end{equation}
Finally, according to~\eqref{eq3powPronu2} and \eqref{eq4powPronu2} we find that the difference $\floor*{\frac{3nm}{7}}-\ceil*{\frac{3nm}{7}}$ holds $\sigma(\CC)>\ceil*{\frac{3nm}{7}}-\floor*{\frac{3nm}{7}}$. Then, by combining this results with~\eqref{eq5powPronu2}, \eqref{eq6powPronu2} and \eqref{eq7powPronu2} we find that~\eqref{eq1powPronu2} holds. As desire.
\end{proof}

Assume $d_T(v_2) = d_T(v_3) = \dots = d_T(v_{n-1}) = k$. Then, the Sigma index satisfies $\sigma(\CC) = d_1(d_1 - k)^2 + d_n(k - d_n)^2$. Sharp upper and lower bounds for topological indices precisely characterize the deviation between these bounds and the actual index values, which constitutes our main objective. Lemma~\ref{ThmMaximumSigman2} establishes the relationship between the harmonic topological index and the Sigma index.

\begin{lemma}~\label{ThmMaximumSigman2}
Let $\CC$ be a caterpillar tree, $H(G)$ be the harmonic topological index and let $\eta$ be an integer, $\mathscr{D}=(d_1,d_2,\dots,d_n)$ be non-decreasing degree sequence with $\lambda_{\mathscr{D}}$ the average of elements $\mathscr{D}$. Then, the upper bound of Sigma index satisfy
\begin{equation}~\label{eq1ThmMaximumSigman2}
\sigma(\CC)\leqslant \floor*{\frac{2n^2}{3\lambda_{\mathscr{D}}}}+\dfrac{2^\eta(m-\Delta)^2}{5(n-1)^3}+(n-1)H(G).
\end{equation}
\end{lemma}
We observe that the conditions $3<\sqrt{nm/2\Delta(\Delta-1)}<4$ and $0<2nm/\Delta(\Delta-1)^2<2$ are also satisfied by the degree sequence $\mathscr{D}$ with $d_1\leqslant d_2 \leqslant \dots \leqslant d_n$. Consequently, these bounds are sharpened by one or more increments. Theorem~\ref{TnpowPronu4} establishes the initial upper bounds by analyzing the sum of squared differences between the degrees in $\mathscr{D}$, while taking into account the effect of the average degree.

\begin{theorem}~\label{TnpowPronu4}
Let $\CC$ be a caterpillar tree, and let $\mathscr{D} = (d_1, d_2, \dots, d_n)$ be a degree sequence where $d_1 \geqslant d_2 \geqslant \dots \geqslant d_n$, with maximum degree $\Delta$, minimum degree $\delta \geqslant 2$, and let $\lambda_{\mathscr{D}}$ be the average of $\mathscr{D}$. Then, the upper bound of the Sigma index satisfies
\begin{equation}~\label{eq1TnpowPronu4}
\sigma(\CC)\leqslant \sum_{i=1}^{n-1}\lambda_{\mathscr{D}}(d_i-d_{i+1})^3 +2(n^2+m^2)+3m+n+2.
\end{equation}
\end{theorem}

\begin{proof}
Assume $\mathscr{D}=(d_1,d_2,\dots,d_n)$ is a degree sequence where $d_1 \geqslant d_2 \geqslant \dots \geqslant d_n$, with maximum degree $\Delta$, minimum degree $\delta \geqslant 2$, and average degree $\lambda_{\mathscr{D}}$. Clearly, $n$ satisfies 
$n > \sum_{i=1}^{n-1}(d_i - d_{i+1})^2$ and $2n > \sum_{i=1}^{n-1}(d_i - d_{i+1})^3.$ 
These $n$-related terms establish the lower bounds for the Sigma index. Then
\begin{equation}~\label{eq2TnpowPronu4}
\sigma(\CC) \geqslant \sum_{i=1}^{n-1}\lambda_{\mathscr{D}} (d_i-d_{i+1})^3+2n^2+2.
\end{equation}
According to proposition~\ref{powPronu2}, we compare the term $2n^3+n^2-2n+1$ with the term $\Delta(\Delta-1)^2$ by considering the value of $m-1,$
\[
0 < \frac{1}{m-1}\left(\frac{2n^3 + n^2 - 2n + 1}{\Delta(\Delta - 1)^2}\right) < 3. 
\]
Based on this reduction, the lower bound~\eqref{eq2TnpowPronu4} is difficult to determine within a certain bound. Thus, we have $\Delta (n - \Delta)^2 - \Delta (m - \Delta)^2 \geqslant \sum_{i=1}^{n-1}\lambda_{\mathscr{D}} (d_i - d_{i+1})^3,$
and for the term $\Delta(\Delta - 1)^2$ compared with the term $(2m^2 + 2m + n) / 2(n + m)$, considering that $\sigma(\CC) < 2m^2 + 2m + n$ holds, we obtain
\[
(d_1 - d_n)^2 \geqslant \frac{1}{\Delta(\Delta - 1)^2} \left(\frac{2m^4 + 2m + n}{2(n + m)} \right).
\]
The sharp lower bound of the Sigma index satisfies the inequality by considering~\eqref{eq2TnpowPronu4} and the value of the term $(d_1 - d_n)^2$,
\begin{equation}~\label{eq3TnpowPronu4}
\sigma(\CC) > \frac{1}{\Delta(\Delta - 1)^2} \left(\frac{2m^4 + 2m + n}{2(n + m)} + 2nm + (d_1 - d_n)^2 \right).
\end{equation}
Thus, relationship~\eqref{eq3TnpowPronu4} implies that $\Delta (2n - m)^2$ satisfies the lower bound according to the value of $\lambda_{\mathscr{D}}$ and the upper bound according to the value $d_i^2$. Hence,
\begin{equation}~\label{eq4TnpowPronu4}
\sum_{i=1}^{n-1} \lambda_{\mathscr{D}} (d_i - d_{i+1})^3 \leqslant \frac{1}{4}\left( \frac{2m^4 + 2m + n}{2 \Delta (n + m)(2n - m)^2} + 2nm + (d_1 - d_n)^2 \right) \leqslant 3 \sum_{i=1}^n d_i^2 - (n + m),
\end{equation}
where
\[
2 \leqslant \frac{1}{4} \left( \frac{2m^4 + 2m + n}{2 \Delta (n + m)(2n - m)^2} \right) \leqslant 3.
\]
Therefore, from~\eqref{eq3TnpowPronu4} and \eqref{eq4TnpowPronu4}, the relationship with~\eqref{eq2TnpowPronu4} implies~\eqref{eq1TnpowPronu4}.
\end{proof}

Actually, by discussing the bounds in~\eqref{eq2TnpowPronu4}, \eqref{eq3TnpowPronu4}, and \eqref{eq4TnpowPronu4}, we obtain an equivalent upper bound to that given in Theorem~\ref{TnpowPronu4}. This bound follows directly from Theorem~\ref{TnpowPronu5} via the relationships in~\eqref{eq2TnpowPronu4}, \eqref{eq3TnpowPronu4}, and \eqref{eq4TnpowPronu4}.

\begin{theorem}~\label{TnpowPronu5}
Let $\CC$ be a caterpillar tree with degree sequence $\mathscr{D}=(d_1,d_2,\dots,d_n)$  where $d_1\geqslant d_2 \geqslant \dots \geqslant d_n$ and $\lambda_{\mathscr{D}}$ the average of $\mathscr{D}$. Then,according to Theorem~\ref{TnpowPronu4} the upper bound of Sigma index satisfy
\begin{equation}~\label{eq1TnpowPronu5}
\sigma(\CC)\leqslant \sum_{i=1}^{n}d_i^3+\frac{2\lambda_{\mathscr{D}}(n^2+m^2)}{3(n+m)}.
\end{equation}
\end{theorem}

The following result is similarly achieved for the upper bound of the Sigma index, making it a sharply distinct and clear extreme value. Therefore, the approximate study to determine a specific sharp value for the Sigma index may not be consistent.
\begin{corollary}~\label{Tercorollaryn1}
Let $\CC$ be a caterpillar tree. Then
\begin{align*}
& \sigma(\CC)\leqslant \sum_{i=1}^{n-1}\lambda_{\mathscr{D}}(d_i-d_{i+1})^3+2m^2+2m+n, \\
& \sigma(\CC)\leqslant \sum_{i=1}^{n-1}\lambda_{\mathscr{D}}(d_i-d_{i+1})^3+2n^2+2.
\end{align*}
\end{corollary}

\begin{corollary}~\label{Tercorollaryn02}
For any tree $T$ of order $n$. Let $F(G)$ be the forgotten topological index. Then, the upper bound of Sigma index satisfy
\begin{equation}~\label{eq1Tercorollaryn02}
\sigma(\CC)\leqslant 2F(G)+2n^2+2.
\end{equation}
\end{corollary}

\begin{proof}
According to the relationship of the forgotten topological index $F(G)$ and according to Theorem~\ref{thmforggotn1}, we noticed that $\sigma(\CC)=F(G)+\sum_{uv \in E(G)}2d_G(u)d_G(v)$. Therefore,
\[
\sum_{uv \in E(G)} (d_G(u)^2 + d_G(v)^2) > \sum_{uv \in E(G)} (d_G(u) - d_G(v))^2.
\]
Thus, $\sigma(\CC)\leqslant 2F(G)$ and the term $2n^2+2$ is confirmed of the upper bound of Sigma index.
\end{proof}

\section{Lower Bounds of Sigma Index among Caterpillar Trees}~\label{sec4}
In this section, The discussion previously presented in section~\ref{sec3} is continued, focusing on the analysis of the lower bound of the Sigma index. Both extreme values play a pivotal role in the behavior of the Sigma index, as illustrated in Figure~\ref{fig001discussion}. The maximum value of Sigma index satisfied with the left figure and the minimum value satisfied with the right figure. Note that the extremal value is attained for the degree sequence ordered by $d_n>d_1>\dots>d_2>d_{n-1}$ and $d_n>d_2>\dots>d_{n-1}>d_1$.

\begin{figure}[H]
    \centering
\begin{tabular}{cc}
\begin{tikzpicture}[scale=.8]
\draw   (1,2)-- (3,2);
\draw   (3,2)-- (5,2);
\draw   (7,2)-- (9,2);
\draw [line width=1pt,dotted] (5,2)-- (7,2);
\draw [line width=1pt,dotted] (3,2)-- (3,0);
\draw [line width=1pt,dotted] (5,2)-- (5,0);
\draw [line width=1pt,dotted] (7,2)-- (7,0);
\draw [line width=1pt,dotted] (9,2)-- (9,0);
\draw [line width=1pt,dotted] (1,2)-- (1,0);
\draw (0.8399938252334064,2.999767639650283) node[anchor=north west] {$d_n$};
\draw (2.816853283710542,2.974744102201205) node[anchor=north west] {$d_1$};
\draw (4.918830429433066,2.9246970273030497) node[anchor=north west] {$d_3$};
\draw (6.745548663215736,2.974744102201205) node[anchor=north west] {$d_2$};
\draw (8.722408121692872,3.0498147145484378) node[anchor=north west] {$d_{n-1}$};
\draw (9.097761183429038,0.6) node[anchor=north west] {$p$};
\draw (7.1709487998500565,0.6) node[anchor=north west] {$p$};
\draw (5.119018729025687,0.6) node[anchor=north west] {$p$};
\draw (3.1421592705485515,0.6) node[anchor=north west] {$p$};
\draw (0.2,0.6) node[anchor=north west] {$p$};
\begin{scriptsize}
\draw [fill=black] (1,2) circle (1.5pt);
\draw [fill=black] (3,2) circle (1.5pt);
\draw [fill=black] (5,2) circle (1.5pt);
\draw [fill=black] (7,2) circle (1.5pt);
\draw [fill=black] (9,2) circle (1.5pt);
\end{scriptsize}
\end{tikzpicture} & \begin{tikzpicture}[scale=.8]
\draw   (1,2)-- (3,2);
\draw   (3,2)-- (5,2);
\draw   (7,2)-- (9,2);
\draw [line width=1pt,dotted] (5,2)-- (7,2);
\draw [line width=1pt,dotted] (3,2)-- (3,0);
\draw [line width=1pt,dotted] (5,2)-- (5,0);
\draw [line width=1pt,dotted] (7,2)-- (7,0);
\draw [line width=1pt,dotted] (9,2)-- (9,0);
\draw [line width=1pt,dotted] (1,2)-- (1,0);
\draw (0.8399938252334064,2.999767639650283) node[anchor=north west] {$d_1$};
\draw (2.816853283710542,2.974744102201205) node[anchor=north west] {$d_{n-1}$};
\draw (4.918830429433066,2.9246970273030497) node[anchor=north west] {$d_3$};
\draw (6.745548663215736,2.974744102201205) node[anchor=north west] {$d_2$};
\draw (8.722408121692872,3.0498147145484378) node[anchor=north west] {$d_n$};
\draw (9.097761183429038,0.6) node[anchor=north west] {$p$};
\draw (7.1709487998500565,0.6) node[anchor=north west] {$p$};
\draw (5.119018729025687,0.6) node[anchor=north west] {$p$};
\draw (3.1421592705485515,0.6) node[anchor=north west] {$p$};
\draw (0.2,0.6) node[anchor=north west] {$p$};
\begin{scriptsize}
\draw [fill=black] (1,2) circle (1.5pt);
\draw [fill=black] (3,2) circle (1.5pt);
\draw [fill=black] (5,2) circle (1.5pt);
\draw [fill=black] (7,2) circle (1.5pt);
\draw [fill=black] (9,2) circle (1.5pt);
\end{scriptsize}
\end{tikzpicture}
    
\end{tabular}
    \caption{The behavior of caterpillar tree $\CC(n,p)$.}
    \label{fig001discussion}
\end{figure}
Through Proposition~\ref{manProtun1}, we establish the lower bound of the Sigma index in accordance with Theorems~\ref{TnpowPronu4} and \ref{TnpowPronu5}, which were also employed to analyze the upper bound discussed earlier. These two propositions introduce new degree-based lower bounds for the Sigma index of caterpillar trees, each emphasizing different structural features of the graph. Proposition \ref{manProtun1} it relates $\sigma(\CC)$ simultaneously to the first Zagreb index $M_1(\CC)$, the average degree $\lambda_{\mathscr D}$, and the parameters $n$ and $m$, yielding a mixed bound that is particularly useful when global degree information and classical indices are available but direct computation of the Sigma index is expensive. Proposition \ref{manProtun2} is providing a purely degree-sequence-based bound expressed in terms of the extremal degrees $d_1,d_n$ and the ratios $d_i/d_{i+1}$, thereby capturing the shape of the ordered degree sequence rather than only its total or maximum. The first proposition is closely connected to existing work on Zagreb indices and average degree in chemical and mathematical graph theory, whereas the second is more directly linked to extremal problems and graph realizations with a fixed degree sequence.

\begin{proposition}~\label{manProtun1}
Let $\CC$ be a caterpillar tree, $\mathscr{D} = (d_1, d_2, \dots, d_n)$ be a degree sequence where $d_1 \geqslant d_2 \geqslant \dots \geqslant d_n$ and $M_1(\CC)$ be the first Zagreb index, and let $\lambda_{\mathscr{D}}$ be the average of $\mathscr{D}$. Then, the lower bound of the Sigma index satisfies
\begin{equation}~\label{eq1manProtun1}
\sigma(\CC)\geqslant M_1(\CC)+ \floor*{\frac{2(n+m)^2}{5}}+\frac{2\lambda_{\mathscr{D}}(n^2+m^2)}{3(n+m)}.
\end{equation}
\end{proposition}
\begin{proof}
Assume $\mathscr{D}=(d_1,d_2,\dots,d_n)$ is a degree sequence where $d_1 \geqslant d_2 \geqslant \dots \geqslant d_n$ with the maximum degree $\Delta$ and $\lambda_{\mathscr{D}}$ is the average of $\mathscr{D}$. Since $2\lambda_{\mathscr{D}}(n^2 + m^2) > 3(n + m)$ and $\Delta < \frac{2\lambda_{\mathscr{D}}(n^2 + m^2)}{3(n + m)} \leqslant 3\Delta$, according to Theorem~\ref{TnpowPronu5} we find the upper bound of the Sigma index given as $\sigma(\CC) \leqslant \sum_{i=1}^{n} d_i^3 + \frac{2\lambda_{\mathscr{D}}(n^2 + m^2)}{3(n + m)}.$
Thus, let us consider the term $\Delta(\Delta - 1)^2$ satisfies $\Delta(\Delta - 1)^2 < 2(m + n)^2$. Then, we find that $\Delta(\Delta - 1)^3 \leqslant 2\lambda_{\mathscr{D}}(n^2 + m^2)$. Thus, the lower bound of the Sigma index satisfies  
\begin{equation}~\label{eq2manProtun1}
\sigma(\CC) \geqslant \Delta(\Delta - 1)^2 + \frac{2\lambda_{\mathscr{D}}(n^2 + m^2)}{3(n + m)}.
\end{equation}
Therefore, from~\eqref{eq2manProtun1} we observe that this bound approaches the Sigma index very closely, due to the values of $\Delta$ that come from the condition we imposed on $\mathscr{D}$ where $d_1 \geqslant d_2 \geqslant \dots \geqslant d_n$. In this case, we note that the term $2(n + m)^2$ satisfies the relation $2(n + m)^2 < \lambda_{\mathscr{D}}(n^2 + m^2)$. Therefore, we obtain the lower bound of the Sigma index as:
\begin{equation}~\label{eq3manProtun1}
\sigma(\CC) \geqslant \floor*{\frac{2(n + m)^2}{5}} + \frac{2\lambda_{\mathscr{D}}(n^2+m^2)}{3(n+m)}.
\end{equation}
Hence, for the last term in~\eqref{eq1manProtun1} we notice that $\sum_{i=1}^{n} d_i^2 < \Delta(\Delta-1)^2$. Thus, according to the lower bounds~\eqref{eq2manProtun1} and \eqref{eq3manProtun1}, the relationship~\eqref{eq1manProtun1} clearly holds. As desired.
\end{proof}

We observe from Proposition~\ref{manProtun1} that it directly leads to Corollary~\ref{Tercorollaryn3}. Also, we notice through Equation~\eqref{eq1Tercorollaryn3} and by directly relying with the term $\sum_{i=1}^{n}d_i^2<\Delta(\Delta-1)^2,$  that it gives us the lower bound of Sigma index.
\begin{corollary}~\label{Tercorollaryn3}
For any tree $T$ of order $n$. The lower bound of Sigma index satisfy
\begin{equation}~\label{eq1Tercorollaryn3}
\sigma(\CC)\geqslant \sum_{i=1}^{n-1}\lambda_{\mathscr{D}}(d_i-d_{i+1})^2+\floor*{\frac{2(n+m)^2}{5}}+\frac{2\lambda_{\mathscr{D}}(n^2+m^2)}{3(n+m)}.
\end{equation}
\end{corollary}

Based on Conjecture~\ref{conjecturen1} which establishes Proposition~\ref{manProtun2}, this conjecture is employed to refine the optimal terms that play a crucial role in enhancing the optimal behavior of the Sigma index.

\begin{conjecture}~\label{conjecturen1}
For any caterpillar tree $\CC$, $k$ be an integer where $0<k<5$. Then, 
\begin{equation}~\label{eq1conjecturen1}
0\leqslant \dfrac{2^k(n+m)^2}{2(k+1)\Delta(\Delta-1)^2} \leqslant 2k.
\end{equation}
\end{conjecture}
\begin{proof}
 Let $k$ be an integer where $0<k<5$, and assume $m=n-1$ in a tree $T$. Then the term $2^k$ with $(n+m)^2$ holds the bound $2^k<(n+m)^2$. Thus, clearly we find that: 
 \begin{equation}~\label{eq2conjecturen1}
 \dfrac{2^k(n+m)^2}{2(k+1)}>0
 \end{equation}
 Therefore, through relationship~\eqref{eq2conjecturen1}, the left-hand side of inequality~\eqref{eq1conjecturen1} is satisfied when considering the term $\Delta(\Delta-1)^2$, which make~\eqref{eq1conjecturen1}  approach zero closely. Thus, for the right-hand side of the inequality~\eqref{eq1conjecturen1}, we find that $(n+m)^2/\Delta(\Delta-1)^2 <2k$. Thus, for the term $2^k(n+m)^2$ is growing up. The right-hand side of inequality~\eqref{eq1conjecturen1} is satisfied.
\end{proof}

Actually,  Proposition \ref{manProtun1} is well suited to situations where robust bounds from aggregate degree data are required, while Proposition \ref{manProtun2} supports fine-grained optimization and comparison of caterpillar trees under prescribed degree-sequence constraints.
\begin{proposition}~\label{manProtun2}
Let $\CC$ be a caterpillar tree and $\mathscr{D}=(d_1,d_2,\dots,d_n)$ be a degree sequence where $d_1 \geqslant d_2 \geqslant \dots \geqslant d_n$. Then, the lower bound of the Sigma index satisfies
\begin{equation}~\label{eq1manProtun2}
\sigma(\CC)\geqslant d_n^3+d_1^3+\sum_{i=1}^{n-1}\frac{d_i}{d_{i+1}}+\floor*{\frac{n^2}{2}}+2n+2.
\end{equation}
\end{proposition}
\begin{proof}
Assume $\mathscr{D}=(d_1,d_2,\dots,d_n)$ is a degree sequence where $d_1\geqslant d_2 \geqslant \dots \geqslant d_n$, and according to Conjecture~\ref{conjecturen1}, we noticed that $0<(n+m)^2/\Delta(\Delta-1)^2\leqslant 2$. Thus, we find that $d_n^3+d_1^3<\Delta(\Delta-1)^2$ and $\sigma(\CC)\leqslant d_n^3+d_1^3+\Delta(\Delta-1)^2$. Then,
\begin{equation}~\label{eq2manProtun2}
\sigma(\CC)\geqslant d_n^3+d_1^3+\Delta(\Delta-1)^2-\floor*{\frac{n^2}{2}}.
\end{equation}
Thus, for $d_i$ and $d_{i+1}$ we noticed that $1\leqslant d_i/d_{i+1}\leqslant 2$. Then, 
\[
\frac{(n+m)^2}{\Delta(\Delta-1)^2}\leqslant \sum_{i=1}^{n-1}\frac{d_i}{d_{i+1}} \leqslant 2n+2,
\]
considering the fact that $0<(n+m)^2/\Delta(\Delta-1)^2\leqslant 2$. Thus, according to~\eqref{eq2manProtun2} we find that
\begin{equation}~\label{eq3manProtun2}
\sigma(\CC)\geqslant d_n^3+d_1^3+\Delta(\Delta-1)^2+\sum_{i=1}^{n-1}\frac{d_i}{d_{i+1}}-\floor*{\frac{n^2}{2}}.
\end{equation}
Therefore, the relationship~\eqref{eq3manProtun2} establishes the lower bound by comparing the terms as
\begin{align*}
 \sigma(\CC) & \geqslant d_n^3+d_1^3+\Delta(\Delta-1)^2+\sum_{i=1}^{n-1}\frac{d_i}{d_{i+1}}-\floor*{\frac{n^2}{2}} \\
 & \geqslant d_n^3+d_1^3+\Delta(\Delta-1)^2+\sum_{i=1}^{n-1}\frac{d_i}{d_{i+1}}-2n-2 \\
 & \geqslant d_n^3+d_1^3+\sum_{i=1}^{n-1}\frac{d_i}{d_{i+1}}+\floor*{\frac{n^2}{2}}.
\end{align*}
Thus, as we know, the value $2n+2$ is close to $2n$. Hence, we find that~\eqref{eq1manProtun2} holds.
\end{proof}

Furthermore, we observed that according to $\mathscr{D}$ where we optimize the terms $(d_i-d_{i+1})^3$ and $(d_i-d_{i+1})^2$. In fact, these terms significantly improve the determination of the lower bounds of the Sigma index, especially when studying the differences according to the given conditions in $\mathscr{D}$.

\begin{lemma}~\label{lemlowSigman1}
 Let $\CC$ be a caterpillar tree with $\mathscr{D}=(d_1,d_2,\dots,d_n)$ be a degree sequence where $d_1 \geqslant d_2 \geqslant \dots \geqslant d_n$. The lower bound of Sigma index satisfies
 \begin{equation}~\label{eq1lemlowSigman1}
\sigma(\CC)\geqslant  \sum_{i=1}^{n-1}(d_i-d_{i+1})^3+ \sum_{i=2}^{n-2}(d_i-d_{i+1})^2+d_1^2+d_n^2+\floor*{\frac{2(n+m)^2}{5}}.
 \end{equation}
\end{lemma}
\begin{proof}
Assume $\mathscr{D}=(d_1,d_2,\dots,d_n)$ is a degree sequence where $d_1 \geqslant d_2 \geqslant \dots \geqslant d_n$. To determine the lower bound of the Sigma index given in equation~\eqref{eq1lemlowSigman1}, we will use mathematical induction. Consequently, we prove the validity of equation~\eqref{eq1lemlowSigman1} for specific values of $n$, namely $n=4,5,6$. 

For $n=4,$ we have
 \begin{equation}~\label{eq2lemlowSigman1}
\sigma(\CC) \geqslant \sum_{i=1}^{n-1}(d_i - d_{i+1})^3 + \sum_{i=2}^{n-2}(d_i - d_{i+1})^2 + d_1^2 + d_n^2 + \floor*{\frac{2(n+m)^2}{5}} - n.
 \end{equation}

For $n=5$ we observe that
  \begin{equation}~\label{eq3lemlowSigman1}
\sigma(\CC) \geqslant \sum_{i=1}^{n-1}(d_i - d_{i+1})^3 + \sum_{i=2}^{n-2}(d_i - d_{i+1})^2 + d_1^2 + d_n^2 + \floor*{\frac{2(n+m)^2}{5}} - 2n.
 \end{equation}

Similarly, for $n=6$ we find
  \begin{equation}~\label{eq4lemlowSigman1}
\sigma(\CC) \geqslant \sum_{i=1}^{n-1}(d_i - d_{i+1})^3 + \sum_{i=2}^{n-2}(d_i - d_{i+1})^2 + d_1^2 + d_n^2 + \floor*{\frac{2(n+m)^2}{5}} - 3n.
 \end{equation}

From~\eqref{eq2lemlowSigman1}, \eqref{eq3lemlowSigman1}, and \eqref{eq4lemlowSigman1}, considering the term with respect to $n$ (since $m = n-1$), the relationship~\eqref{eq1lemlowSigman1} holds true for $n=4,5,6$. 

Assume the relationship is true for $n$; we must prove it for $n+1$. Therefore,
 \begin{align*}
\sigma(\CC) & \geqslant \sum_{i=1}^{n-1}(d_i - d_{i+1})^3 + \sum_{i=2}^{n-2}(d_i - d_{i+1})^2 + d_1^2 + d_n^2 + \floor*{\frac{2(n+m)^2}{5}} \\
& \geqslant \sum_{i=2}^{n-2}(d_i - d_{i+1})^3 + \sum_{i=1}^{n-3}(d_i - d_{i+1})^2 + d_1^2 + d_{n-1}^2 + d_n^2 + \floor*{\frac{2(n+m-2)^2}{5}} \\
& \geqslant \sum_{i=2}^{n}(d_i - d_{i+1})^3 + \sum_{i=3}^{n-1}(d_i - d_{i+1})^2 + d_1^2 + d_n^2 + \floor*{\frac{2(n+m-2)^2}{5}} \\
& \geqslant \sum_{i=3}^{n+1}(d_i - d_{i+1})^3 + \sum_{i=4}^{n}(d_i - d_{i+1})^2 + d_1^2 + d_n^2 + \floor*{\frac{2(n+m-3)^2}{5}}.
 \end{align*}

Note that the constant term is $\floor*{\frac{2(n+m)^2}{5}}$ and since $(d_i - d_{i+1})^3$ is positive due to the non-increasing order $d_1 \geqslant d_2 \geqslant \dots \geqslant d_n$, we have
  \begin{equation}~\label{eq5lemlowSigman1}
\sigma(\CC) \geqslant n(d_i-d_{i+1})_{i \geqslant 1}^3+n(d_i-d_{i+1})_{i \geqslant 2}^2+d_1^2+d_n^2+\floor*{\frac{2(n+m-i)_{i \leqslant m}^2}{5}}.
 \end{equation}

Finally, from~\eqref{eq4lemlowSigman1} and~\eqref{eq5lemlowSigman1}, we conclude that the lower bound of the Sigma index holds true for all $n+1$ and $n+2$. Thus, equation~\eqref{eq1lemlowSigman1} is established.
\end{proof}

Similarly, according to Lemma~\ref{lemlowSigman1} we obtained the following result such that we optimized the constant term $2(n+m-i)_{i\geqslant 2}^2$ with $n(d_i-d_{i+1})_{i \geqslant 1}^3$ and $n(d_i-d_{i+1})_{i \geqslant 2}^2$.
\begin{corollary}~\label{Tercorollaryn4}
For any tree $T$ with $\mathscr{D}=(d_1,d_2,\dots,d_n)$ be a degree sequence where $d_1\geqslant d_2 \geqslant \dots \geqslant d_n$.  The lower bound of Sigma index satisfy
 \begin{equation}~\label{eq1Tercorollaryn4}
\sigma(\CC)\geqslant  \sum_{i=1}^{n-1}(d_i-d_{i+1})^4+ \sum_{i=2}^{n-2}(d_i-d_{i+1})^3+d_1^2+d_n^2+\floor*{\frac{2(n+m-2)^2}{5}}.
 \end{equation}
\end{corollary}

Throughout Corollary~\ref{Tercorollaryn4}, we analyze a specific case derived from Lemma~\ref{lemlowSigman1}. Taken together, these results yield an improved lower bound for the Sigma index, which is sufficiently close to the Sigma index itself and provides a sharp, explicit value with respect to $\mathscr{D}$. Accordingly, we take the constant term to be $\floor*{\frac{2n-1}{2}}\ceil*{\frac{3n-2}{3}}$, determined by the values of $n$ and $\Delta$.

\section{Discussion The Effects of Bounds on Behavior of Sigma Index}~\label{sec5}
The results presented in Sections~\ref{sec3} and \ref{sec4} show that the optimal behavior of the Sigma index is analyzed by determining its upper and lower bounds together with those of several selected topological indices.
 Thus, the upper bound of the Sigma index as
\begin{equation}~\label{eqqupeerbomn1}
\sigma(\CC)\leqslant \floor*{\frac{3n}{5}} \ceil*{\frac{2n}{5}} \frac{2nm}{\Delta(\Delta-1)^2}+\delta \sqrt{nm/2\Delta(\Delta-1)}+\frac{1}{2}(n^2+m^2)+n(m-\Delta).
\end{equation}
Also, the upper bound of the Sigma index is
\[
\sigma(\CC)\leqslant \sum_{i=1}^{n-1}\lambda_{\mathscr{D}}(d_i-d_{i+1})^3 +2(n^2+m^2)+3m+n+2.
\]
Hence, we observe that the difference making the upper bound sufficiently close is $n^2+m^2$ and $n(m-\Delta)$. Thus, we presented among Theorem~\ref{TnpowPronu5}, the relationship of the sharp upper bound of Sigma index with it collories as
\[
\sigma(\CC)\leqslant \sum_{i=1}^{n}d_i^3+\frac{2\lambda_{\mathscr{D}}(n^2+m^2)}{3(n+m)}.
\]
Thus, let $k$ be an integer where $2\leqslant k <m-1$. Then, the upper bound of Sigma index associated with the terms $n^2+m^2$ and $n(m-\Delta)$. Therefore, the closest upper bound of Sigma index satisfy
\begin{equation}~\label{eq1Discussion}
\sigma(\CC) \leqslant k(n^2+m^2)+n(km-\Delta).
\end{equation}
From another perspective, based on PropositionS~\ref{manProtun1}, \ref{manProtun2} and Lemma~\ref{lemlowSigman1}, the fixed constant in each that ensures the lower bound closely approximates the Sigma index related to the term $\floor*{(2(n+m)^2)/5}$. Consequently, by considering relation~\eqref{eq1Discussion} and employing a similar reasoning, we derive the value that most accurately approaches the Sigma index through the tightest possible lower bound as:
\begin{equation}~\label{eq2Discussion}
\sigma(\CC) \geqslant \frac{1}{k}(n^2+m^2)+\frac{1}{m-k}n(m-\Delta)+\floor*{\frac{2(n+m)^2}{5k}}.
\end{equation}
Furthermore, from~\eqref{eq1Discussion} and \eqref{eq2Discussion} we noticed that the difference between both bounds
\[
k(n^2+m^2)+n(km-\Delta)-\left( \frac{1}{k}(n^2+m^2)+\frac{1}{m-k}n(m-\Delta)+\floor*{\frac{2(n+m)^2}{5k}}\right)
\]
yield to 
\[
\frac{k^2-1}{k}(n^2+m^2)+n\left(km-\Delta-\frac{m-\Delta}{m-k}\right)-\floor*{\frac{2(n+m)^2}{5k}}
\]
Therefore, the difference provides the upper bound of Sigma index which we noticed that: 
\[
\sigma(\CC) < \frac{k^2-1}{k}(n^2+m^2)+n\left(km-\Delta-\frac{m-\Delta}{m-k}\right)-\floor*{\frac{2(n+m)^2}{5k}}.
\]
To enhance the previous discussion, we compare the behaviors of these extreme bounds. The lower bound 1 (LB1) is given by Lemma~\ref{lemlowSigman1}, and the lower bound 2 (LB2) by~\eqref{eq2Discussion}. Similarly, the upper bound 1 (UB1) is from~\eqref{eqqupeerbomn1}, and the upper bound 2 (UB2) from Theorem~\ref{TnpowPronu4}.

The positions labeled $n, \sigma(\CC)$, LB1, LB2, UB1 and UB2 correspond to the variables whose pairwise correlation coefficients are displayed in the correlation matrix (see Figure~\ref{fig002discussion}) as 
\[
M_{\mathscr{D}}=
\begin{pmatrix}
1.000000 & 0.954243 & 0.985349 & 0.928820 & 0.956584 & 0.959995\\
0.954243 & 1.000000 & 0.991001 & 0.996741 & 0.999966 & 0.999754\\
0.985349 & 0.991001 & 1.000000 & 0.977349 & 0.992055 & 0.993259\\
0.928820 & 0.996741 & 0.977349 & 1.000000 & 0.996057 & 0.995292\\
0.956584 & 0.999966 & 0.992055 & 0.996057 & 1.000000 & 0.999862\\
0.959995 & 0.999754 & 0.993259 & 0.995292 & 0.999862 & 1.000000
\end{pmatrix}
\]
Noticed that, regression coefficients $[\,-7.047368 \, , -4.26834614\,]$,  intercept $13518.783306538098$, Model $R^2$ is  $0.9999999997765274$.

\begin{figure}[H]
    \centering
    \includegraphics[width=0.9\linewidth]{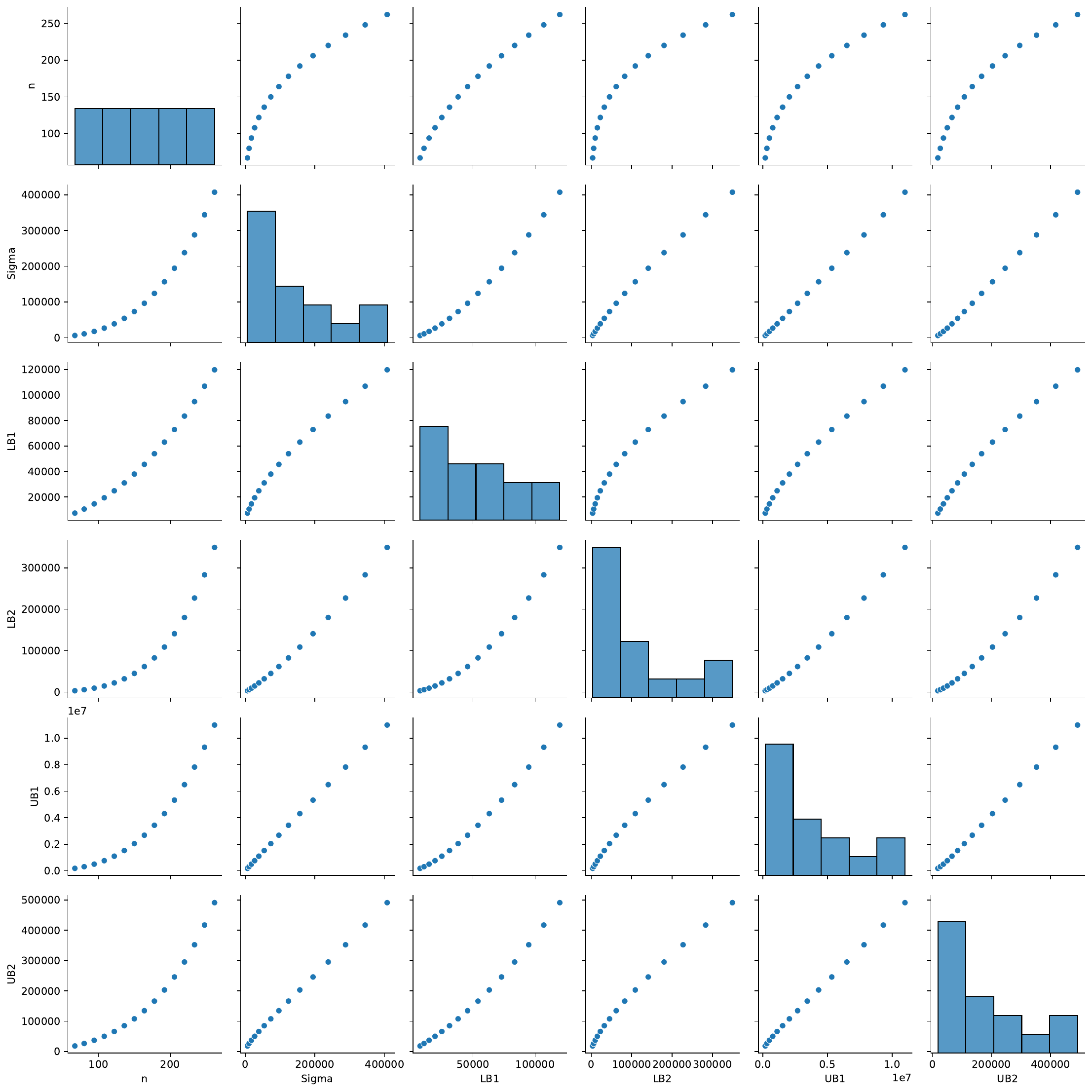}
    \caption{Analysis the extremal value of Sigma index.}
    \label{fig002discussion}
\end{figure}
\section{Conclusion}
This paper investigates bounds on the Sigma index and their connections to harmonic topological indices, as shown in Theorem~\ref{ThmMaximumSigman2}. There, we established an upper bound of
$$
\sigma(\CC)\leqslant \floor*{\frac{2n^2}{3\lambda_{\mathscr{D}}}}+\dfrac{2^\eta(m-\Delta)^2}{5(n-1)^3}+(n-1)H(G).
$$
Corollary~\ref{Tercorollaryn02} further relates the forgotten topological index to the Sigma index, giving $\sigma(\CC)\leqslant 2F(G)+2n^2+2$. These findings illustrate key extremal values affecting the Sigma index through upper and lower bounds. Lemma~\ref{powPronu2} refine these upper bounds, while Theorems~\ref{TnpowPronu4} and~\ref{TnpowPronu5} offer a particularly tight and uniform result. In Section~\ref{sec4}, we examine lower bound improvements, supported by Lemma~\ref{lemlowSigman1} and Propositions~\ref{manProtun1} and~\ref{manProtun2}. Such analysis of extremal bounds on the Sigma index highlights their potential significance for topological indices in general.

\end{document}